\documentclass[11pt]{amsart}
\usepackage{amsfonts,amssymb,amsmath,amsthm}
\usepackage{url}
\usepackage{enumerate}
\usepackage[dvipdfmx]{graphicx}

\newtheorem{theorem}{Theorem}[section]

\newtheorem{lemma}[theorem]{Lemma}

\newtheorem{proposition}[theorem]{Proposition}
\newtheorem{corollary}[theorem]{Corollary}
\theoremstyle{definition}
\newtheorem{definition}[theorem]{Definition}
\newtheorem{remark}[theorem]{Remark}
\numberwithin{equation}{section}

\author{Shohei Satake}
\address{
Graduate School of System Informatics, Kobe University \\
Rokkodai 1-1, Nada, Kobe, 657-8501, JAPAN}
\email{155x601x@stu.kobe-u.ac.jp}
\thanks{The author is supported by Grant-in-Aid for JSPS Fellows 18J11282 of the Japan Society for the Promotion of Science.}



\begin{document}

\title[On Explicit Random-like Tournaments]
{On Explicit Random-like Tournaments}

\maketitle

\begin{abstract}
We give a new theorem describing a relation between the quasi-random property of regular tournaments and their spectra.
This provides many solutions to a constructing problem mentioned by Erd\H{o}s and Moon (1965) and Spencer (1985).
\end{abstract}

\section{Introduction}
\label{intro}
A {\it tournament} is an oriented complete graph.
{\it Random tournaments} $\mathcal{T}_n$ with $n$ vertices are obtained by choosing a direction of each edge of a complete graph with $n$ vertices with probability $1/2$, independently.
We say that random tournaments {\it asymptotically almost surely} ({\it a.a.s.}) satisfy a property $\mathcal{P}$ if 
the probability of the event that tournaments satisfy $\mathcal{P}$ tends to $1$ when $n$ goes to infinity.
In graph theory, there have been many problems focusing on deterministic tournaments satisfying properties which random tournaments a.a.s satisfy; see e.g. \cite{AS16}, \cite{B09}, \cite{CG91}, \cite{CR17}, \cite{KS13}.
\par In this paper, as such a property, we mainly focus on the {\it quasi-random property} proposed by Chung-Graham~\cite{CG91}. 
Our main result is to give a new theorem describing a relation between the quasi-random property and spectra of regular tournaments.
This result also provides many solutions to a problem, proposed by Erd\H{o}s-Moon~\cite{EM65} and Spencer~\cite{S85} (see also \cite[Section 9.1]{AS16}),
on explicit constructions of tournaments with a small number of consistent edges. 
It is well-known that Paley tournaments have the quasi-random property (e.g. \cite{CG91}). Moreover, by proving that Paley tournaments have a property stronger than the quasi-random property, Alon-Spencer~\cite{AS16} showed that they provide solutions to the problem by Erd\H{o}s, Moon and Spencer.
We note that the proof in \cite{AS16} contains a part (Lemma 9.1.2 in \cite{AS16}) depending on the definition of Paley tournaments. 
Remarkably, we generalize their discussion to all regular tournaments by using a digraph-version of the {\it expander-mixing lemma} proved by Vu~\cite{V08}.
\par The rest of this paper is organized as follows.
In Section~\ref{sect:Rank}, we recap the quasi-random property and introduce some related known facts.
In Section~\ref{sect:Main}, we introduce our main result and give its proof. 
In Section~\ref{sect:Const}, we provide some examples of regular tournaments satisfying the quasi-random property which are also solutions to the problem by Erd\H{o}s, Moon and Spencer. 
At last, in Section~\ref{sec:Shutte}, we discuss another random-like property defined as an adjacency property.

\section{The quasi-random property and related facts}
\label{sect:Rank}
In this section, we review the quasi-random property and some related known facts.
For a digraph $D$, let $V(D)$ and $E(D)$ be the vertex and the edge set of $D$, respectively. 
For two distinct vertices $x$ and $y$, let the ordered pair $(x, y)$ denote the edge directed from $x$ to $y$.
\par First, we give the definition of the quasi-random property of tournaments which was formulated by Chung-Graham~\cite{CG91}.
\begin{definition}[The quasi-random property, \cite{CG91}]
\label{def:quasi}
Let $T$ be a tournament with $n$ vertices.
Let $\sigma$ be a bijection from $V(T)$ to $\{1, 2, \ldots, n\}$.
An edge $(x, y)$ of $T$ is called {\it consistent} with $\sigma$ if $\sigma(x)<\sigma(y)$. 
Let $C(T, \sigma)$ be the number of consistent edges with $\sigma$ and $C(T)=\max_{\sigma}C(T, \sigma)$.
Then, $T$ has the {\it quasi-random property} if $T$ satisfies 
\begin{equation}
\label{eq:quasi}
C(T) \leq (1+o(1))\frac{n^2}{4}.
\end{equation}
\end{definition}
Surprisingly, Chung-Graham~\cite{CG91} gave some other properties which are seemingly unrelated, but actually equivalent with (\ref{eq:quasi}). 
The interested reader is referred to \cite{CG91}. 
\par Consistent edges of tournaments was originally investigated by Erd\H{o}s-Moon~\cite{EM65}.
Their work was from paired comparisons (e.g. \cite{KS40}). 
It is reasonable to find suitable rankings, that is, bijections with many consistent edges.
First observe that for every tournament $T$ with $n$ vertices,
\begin{equation}
   \frac{1}{2}\binom{n}{2} \leq C(T)\leq \binom{n}{2}. 
\end{equation}
The lower bound of $C(T)$ is obtained by the following simple fact$\colon$
\begin{equation}
\label{eq:plus}
C(T, \sigma)+C(T, \sigma')=\binom{n}{2},
\end{equation}
where $\sigma'$ is the reversed ranking of $\sigma$ which is defined as $\sigma'(v)=n+1-\sigma(v)$ for each $v \in V(T)$.
For the upper bound of $C(T)$, the equality holds if and only if $T$ is a transitive tournament.
On the other hand, it is non-trivial to check the tightness of the lower bound of $C(T)$.
In \cite{EM65}, it was proved that there exist tournaments $T$ such that $C(T)\leq (1+o(1))\binom{n}{2}/2$ by a probabilistic argument.
Moreover Spencer~\cite{S71}, \cite{S80} and de la Vega~\cite{d83} proved that random tournament $\mathcal{T}_n$ a.a.s satisfies the following property which is stronger than the quasi-random property$\colon$
\begin{equation}
\label{eq:stongquasi}
C(\mathcal{T}_n) \leq \frac{1}{2}\binom{n}{2}+O(n^{\frac{3}{2}}).
\end{equation}
\par Erd\H{o}s-Moon~\cite{EM65} and Spencer~\cite{S85} mentioned the problem on explicit constructions of tournaments $T$ such that $C(T)$ is close to the lower bound. 
At present, such a construction of tournaments $T$ giving the best known ^^ ^^ constructive" upper bound of $C(T)$ is obtained by Alon-Spencer~\cite{AS16}.
For a prime $p \equiv 3 \pmod{4}$, the {\it Paley tournament} $T_p$ is the tournament with vertex set $\mathbb{F}_p$, the finite field of $p$ elements, 
and edge set formed by all edges $(x, y)$ such that $x-y$ is a non-zero square of $\mathbb{F}_p$. 
In \cite[Theorem 9.1.1]{AS16}, it was proved that 
\begin{equation}
\label{eq:paley}
C(T_p) \leq \frac{1}{2}\binom{p}{2}+O(p^{\frac{3}{2}}\log p).
\end{equation}
In Section~\ref{sect:Const}, by applying the main theorem proved in the next section, we give some new explicit constructions of regular tournaments $T$ with $n$ vertices such that $C(T)$ is close to the lower bound.

\section{Main theorem}
\label{sect:Main}
In this section, we prove our main theorem.
We first give the definition of regular digraphs and the adjacency matrix of a digraph.
A digraph is said to be {\it $d$-regular} if in-degree and out-degree of each vertex is $d$.
Especially a tournament with $n$ vertices is simply said to be {\it regular} if it is $(n-1)/2$-regular. 
The {\it adjacency matrix} $M_D$ of a digraph $D$ with vertices is the $\{0, 1\}$-square matrix of size $n$ 
whose rows and columns are indexed by the vertices of $D$ and the $(x, y)$-entry is equal to $1$ if and only if $(x, y) \in E(D)$.

The following is our main theorem.
\begin{theorem}
\label{thm:main}
Let $T$ be a regular tournament with $n$ vertices.
Suppose that the adjacency matrix $M_T$ of $T$ has eigenvalues such that $(n-1)/2=\lambda_1, \lambda_2, \cdots, \lambda_{n}$. Let $\lambda(T)=\max_{2\leq i \leq n}|\lambda_i|$.
Then, 
\begin{equation}
C(T)\leq \frac{1}{2}\binom{n}{2}+\lambda(T) \cdot n \log_2 (2n).
\end{equation}
\end{theorem}

\begin{remark}
Theorem~\ref{thm:main} implies that every regular tournament $T$ with $n$ vertices such that $\lambda(T)=o(n/\log n)$ has the quasi-random property.
It should be remarked that Kalyanasundaram-Shapira~\cite{KS13} shows a stronger result; a proof of Lemma 2.3 and the first concluding remark in \cite{KS13} implies that a regular tournament $T$ with $n$ vertices has the quasi-random property if and only if $T$ satisfies that $\lambda(T)=o(n)$.
(In \cite{KS13}, the authors considered the eigenvalues of the $\{0, \pm 1\}$-matrix $2M_T-J_n+I_n$, but these eigenvalues can be directly computed from ones of $M_T$.) 

On the other hand, Theorem~\ref{thm:main} not only gives a spectral condition for the quasi-random property, but also implies that estimating eigenvalues of $M_T$ provides better upper bounds of $C(T)$ than the bound (\ref{eq:quasi}).
Thus, considering (\ref{eq:stongquasi}), Theorem~\ref{thm:main} provides a spectral condition for a property, which random tournaments a.a.s. satisfy, stronger than the quasi-random property;
for example, if $T$ satisfies $\lambda(T)=o(n/\log n)$, then Theorem~\ref{thm:main} implies that $C(T) \leq \binom{n}{2}/2+o(n^2)$, which immediately implies the quasi-random property.
\end{remark}

In the proof of Theorem~\ref{thm:main}, we use the {\it expander-mixing lemma} for normal regular digraphs proved by Vu~\cite{V08}.
A digraph $D$ is said to be {\it normal} if $M_D$ and its transpose $M_D^t$ are commutative.
In other word, $D$ is normal if $|N^+(x,y)|=|N^-(x,y)|$ for any two distinct vertices $x$ and $y$ where $N^+(x, y)$ (resp. $N^-(x, y)$) is the set of vertices $z$ 
such that $(x, z) , (y, z) \in E(D)$ (resp. $(z, x) , (z, y) \in E(D)$).

Now we are ready to introduce the expander-mixing lemma for normal regular digraphs.

\begin{lemma}[Expander-mixing lemma, \cite{V08}]
\label{lem:exp} 
Let $D$ be a normal $d$-regular digraph with $n$ vertices and $\lambda(D)=\max_{2\leq i \leq n}|\lambda_i|$. 
For two disjoint subsets $A, B \subset V(D)$, let
\[
e(A, B):=\bigl|\{(a, b)\in E(D) \mid a \in A, \; b \in B\} \bigr|.
\]
Then for every pair of two disjoint subsets $A, B \subset V(D)$, it holds that
\begin{align}
\Bigl|e(A, B)-\frac{d}{n} \cdot |A|\cdot |B|  \Bigr| \leq \lambda(D)\sqrt{|A|\cdot |B|}.
\end{align} 
\end{lemma}
From this lemma, we can easily obtain the following corollary. 
\begin{corollary}
\label{cor:exp}
Let $D$ be a normal $d$-regular digraph with $n$ vertices. 
Then for every pair of two disjoint subsets $A, B \subset V(D)$, 
\begin{align}
|e(A, B)-e(B, A)| \leq 2\lambda(D)\sqrt{|A|\cdot |B|}.
\end{align}
\end{corollary}
\begin{proof}
From the triangle inequality, we see that 
\begin{align*}
|e(A, B)-e(B, A)|&=\Bigl| \Bigl(e(A, B)-\frac{d}{n} \cdot |A|\cdot |B|\Bigr)-\Bigl(e(B, A)-\frac{d}{n} \cdot |B|\cdot |A|\Bigr) \Bigr|\\
&\leq \Bigl|e(A, B)-\frac{d}{n} \cdot |A|\cdot |B| \Bigr|+\Bigl|e(B, A)-\frac{d}{n} \cdot |B|\cdot |A|\Bigr|.
\end{align*}
Thus, by Lemma~\ref{lem:exp}, we get the corollary.
\end{proof}

By Corollary~\ref{cor:exp}, we get the following lemma.
\begin{lemma}
\label{cor:exp2}
Let $T$ be a regular tournament with $n$ vertices and let $\sigma$ be a bijection from $V(T)$ to $\{1, 2, \ldots, n\}$. 
Then
\begin{equation}
C(T, \sigma)-C(T, \sigma') \leq  2\lambda(T) \cdot n \log_2 (2n).
\end{equation}
\end{lemma}
\begin{proof}[Proof of Lemma~\ref{cor:exp2}]
The lemma follows by combining Corollary~\ref{cor:exp} and the argument in \cite[pp.150-151]{AS16} to prove the bound (\ref{eq:paley}) for Paley tournaments. 
It should be noted (see also \cite{BG72}) that every regular tournament $T$ with $n$ vertices is normal since it holds that $M_T^t=J_n-I_n-M_T$, where $I_n$ and $J_n$ are the identity matrix and the all-one matrix of order $n$, respectively.

\par \par Fix a bijection $\sigma$. Let $r$ be the smallest integer such that $n \leq 2^r$.
Let $n=a_1+a_2$, where $a_1$ and $a_2$ are positive integers with $a_1, a_2 \leq 2^{r-1}$. 
Consider a partition of $V(T)$, say $A_1$ and $A_2$, such that $A_1$ is the set of ^^ ^^ highly ranked" $a_1$ vertices in $\sigma$ and $A_2$ is the remaining $a_2$ vertices.
It follows from Corollary~\ref{cor:exp} that 
\begin{align}
\label{eq2:thm2}
e(A_1, A_2)-e(A_2, A_1) \leq 2\lambda(T)\sqrt{a_1 a_2} \leq 2\lambda(T) \cdot 2^{r-1}.
\end{align}
Next, let $a_1=a_{11}+a_{12}$, where $a_{11}$ and $a_{12}$ are positive integers with $a_{11}, a_{12} \leq 2^{r-2}$,
and similarly for $a_2=a_{21}+a_{22}$.
As above, divide $A_1$ into two subsets, say $A_{11}$ and $A_{12}$, where $A_{11}$ is the set of ^^ ^^ highly ranked" $a_{11}$ vertices of $A_1$ in $\sigma$
and $A_{12}$ is the remaining $a_{12}$ vertices of $A_1$.
For $a_{21}$ and $a_{22}$, two subsets $A_{21}$ and $A_{22}$ of $A_2$ are defined in the same way as $A_{11}, A_{12}$.
It then follows from Corollary~\ref{cor:exp} that
\begin{align*}
&e(A_{11}, A_{12})-e(A_{12}, A_{11}) + e(A_{21}, A_{22})-e(A_{22}, A_{21}) \\
&\leq 2\lambda(T) \sqrt{a_{11} a_{12}}+2\lambda(T) \sqrt{a_{21} a_{22}} \\
&\leq 2 \cdot 2\lambda(T) \cdot 2^{r-2}.
\end{align*}
Then iterate such estimation from the first to the $r$-th step. 
In the $i$-th step, $V(T)$ is partitioned into $2^i$ subsets, say $A_{\boldsymbol{\varepsilon}1}$ and $A_{\boldsymbol{\varepsilon}2}$ ($\boldsymbol{\varepsilon} \in \{1, 2\}^{i}$), such that each $A_{\boldsymbol{\varepsilon}j}$ ($j=1, 2$) contains at most $2^{r-i}$ vertices which are consecutive in $\sigma$.  
It follows from Corollary~\ref{cor:exp} that
\begin{align}
\label{eq3:thm2}
\sum_{\boldsymbol{\varepsilon} \in \{1, 2\}^{i-1}}\{e(A_{\boldsymbol{\varepsilon} 1}, A_{\boldsymbol{\varepsilon}2})-e(A_{\boldsymbol{\varepsilon}2}, A_{\boldsymbol{\varepsilon}1})\} 
\leq 2^{i-1} &\cdot 2\lambda(T) \cdot 2^{r-i}=2\lambda(T) \cdot2^{r-1}.
\end{align}

On the other hand, it turns out from the construction of partitions that 
\begin{align}
\label{eq4:thm2}
\sum_{1 \leq i \leq r}\sum_{\boldsymbol{\varepsilon}\in \{1, 2\}^{i-1}}\{e(A_{\boldsymbol{\varepsilon} 1}, A_{\boldsymbol{\varepsilon}2})-e(A_{\boldsymbol{\varepsilon}2}, A_{\boldsymbol{\varepsilon}1})\}
=C(T, \sigma)-C(T, \sigma').
\end{align}
Thus by combining (\ref{eq3:thm2}) and (\ref{eq4:thm2}), it follows that
\[
C(T, \sigma)-C(T, \sigma') \leq r \cdot 2\lambda(T) \cdot 2^{r-1} \leq 2\lambda(T) \cdot n \log_2 (2n).
\] 
\end{proof}

\begin{proof}[Proof of Theorem~\ref{thm:main}]
The theorem is a direct consequence of the equality (\ref{eq:plus}) and Lemma~\ref{cor:exp2}.
\end{proof}

\begin{remark}
It should be noted that for every regular tournament $T$ with $n$ vertices, $\lambda(T) \cdot n \log_2 (2n)$ cannot be less than $\sqrt{n^{3}+n} \log_2 (2n)/2$.
In fact, for every such tournament $T$, it holds that
\begin{equation}
\label{eq:lambda}
    \lambda(T) \geq \frac{\sqrt{n+1}}{2}.
\end{equation}
Indeed, for every strongly-connected normal $d$-regular digraph $D$ with $n$ vertices, it holds that
\[
nd = E(D)= Tr(M_DM_D^t) =\sum_{i=1}^n|\lambda_i|^2 \leq d^2+(n-1)\lambda(D)^2,
\]
which follows from the hand shaking lemma and the Perron-Frobenius theorem (see e.g. \cite{LW01}).
The idea of the above inequality can be found in \cite[p.217]{KS06}.
Also note that every regular tournament $T$ is strongly connected, which follows from the Perron-Frobenius theorem and facts that $T$ is normal and every eigenvalue of $M_T$ corresponding to eigenvectors distinct to the all-one vector has the real part equal to $-1/2$ (see also \cite{BG68}).
\end{remark}

\section{Examples of quasi-random regular tournaments}
\label{sect:Const}
In this section, we give some examples of regular tournaments $T$ with $n$ vertices and $\lambda(T)=o(n/\log n)$. 
As will be shown below, we can construct such tournaments for almost all positive integers $n$.

\par First we consider the following tournaments constructed from finite fields which are variants of cyclotomic tournaments (see e.g. \cite{S16} and reference therein).
Let $m$ be a positive even integer and $p\equiv m+1 \pmod {2m}$ be a prime.
Note that there exist infinitely many such primes by the Dirichlet's theorem on arithmetic progressions and the fact that $m+1$ and $2m$ are coprime when $m$ is even.   
Recall that $\mathbb{F}_p$ is the finite field of order $p$. 
Let $g$ be a primitive element of $\mathbb{F}_p$.
For even $m$, the multiplicative group of $\mathbb{F}_p$, which is denoted by $\mathbb{F}_p^*$, is divided into $m$ cosets $S_0, S_1, \ldots, S_{m-1}$ where  
$S_i:=\{g^t \mid t \equiv i \pmod{m} \}$ for each $0 \leq i \leq m-1$.
Note that $S_j = -S_i$ if $j \equiv -i \pmod{m}$.
 
\begin{definition}
Let $\boldsymbol{i}=(i_1, i_2, \ldots, i_{m/2}) \in \{0, 1, \ldots, m-1\}^{m/2}$ such that $S_{\boldsymbol{i}}=S_{i_1} \cup \cdots \cup S_{i_{m/2}}$ and $\mathbb{F}_p^* \setminus S=-S$.
Then the tournament $T_p^m(S_{\boldsymbol{i}})$ is defined as follows:
\begin{equation}
\begin{split}
&V(T_p^m(S_{\boldsymbol{i}}))=\mathbb{F}_p, \\
&E(T_p^m(S_{\boldsymbol{i}}))=\{(x, y)\in \mathbb{F}_p^2 \mid x-y \in S_{\boldsymbol{i}} \}.
\end{split}
\end{equation}
\end{definition}

This is a direct generalization of Paley tournament since $T_p^m(S_{\boldsymbol{i}})$ is exactly $T_p$ in the case of $m=2$.
Moreover from the definition, it is not so hard to see that $T_p^m(S_{\boldsymbol{i}})$ is a regular tournament with $p$ vertices.
\par Now we obtain the following corollary.

\begin{corollary}
\label{cor:Tpm}
\begin{equation}
C\bigl(T_p^m(S_{\boldsymbol{i}})\bigr) \leq \frac{1}{2}\binom{p}{2}+O(p^{\frac{3}{2}} \log p).
\end{equation}
\end{corollary}

Corollary~\ref{cor:Tpm} is proved by combining Lemma~\ref{cor:exp2} and the following evaluation of $\lambda(T_p^m(S_{\boldsymbol{i}}))$.

\begin{lemma}
\label{lem:eigen}
\begin{align}
\lambda \bigl(T_p^m(S_{\boldsymbol{i}})\bigr) \leq \frac{m\sqrt{p}}{2}.
\end{align}
\end{lemma}
\begin{proof}
First, by a simple calculation, it can be shown that the set of eigenvalue of $M_{T_p^m(S_{\boldsymbol{i}})}$ is
\[
\Bigl\{ \sum_{s\in S_{\boldsymbol{i}}}\psi(s) \mid \text{$\psi$ is an additive character of $\mathbb{F}_p$} \Bigr\}.
\]
Since $S_{i}=g^{i}S_0$ for each $1 \leq i \leq m-1$, we see that 
\begin{align}
\label{eq:eigen1}
\sum_{s\in S_i}\psi(s)=\sum_{s\in g^iS_0}\psi(s)=\sum_{s\in S_0}\psi(g^i s).
\end{align}
Since $S_0$ is the set of non-zero $m$-th power elements and 
each non-zero $m$-th power residue appears exactly $m$ times in the sequence $(x^m)_{x \in \mathbb{F}_{p}^*}$,  
\begin{align}
\label{eq:eigen2}
\sum_{s\in S_0}\psi(g^i s)=\frac{1}{m}\sum_{x \in \mathbb{F}_p^*}\psi(g^i x^m).
\end{align}
At last, we use the following known estimation (see e.g. \cite[p.44]{S76});
\begin{align}
\label{eq:eigen3}
\Bigl|\sum_{x \in \mathbb{F}_p}\psi(a x^m) \Bigr| \leq (m-1)\sqrt{p},
\end{align}
for any non-trivial additive character $\psi$ and $a \neq 0$. 
By combining (\ref{eq:eigen1}), (\ref{eq:eigen2}) and (\ref{eq:eigen3}), 
\begin{align*}
\lambda \bigl(T_p^m(S_{\boldsymbol{i}})\bigr)  \leq \frac{m}{2}\cdot \frac{1}{m} \cdot \{(m-1)\sqrt{p}+1\}
=\frac{(m-1)\sqrt{p}+1}{2} \leq \frac{m\sqrt{p}}{2}.
\end{align*}
\end{proof}

The second example is doubly regular tournament which has been extensively studied in algebraic combinatorics and related areas (e.g. \cite{RB72}).
\begin{definition}
A tournament $T$ with $n$ vertices is called a {\it doubly regular tournament} if $T$ is a regular tournament such that for any distinct two vertices $x$ and $y$, $N^+(x, y)=N^-(x, y)=(n-3)/4$.
\end{definition}
Let $DRT_n$ denote a doubly regular tournament with $n$ vertices.

\begin{corollary}
\label{cor:dr}
\begin{equation}
C(DRT_n)\leq \frac{1}{2}\binom{n}{2}+O(n^{\frac{3}{2}} \log n).
\end{equation}
\end{corollary}
Corollary~\ref{cor:dr} is proved by the following well-known evaluation of $\lambda(DRT_n)$ which also shows that the inequality (\ref{eq:lambda}) is tight.
\begin{lemma}[e.g. \cite{dGKPM92}]
\label{lem:eigen2}
\begin{align}
\lambda(DRT_n)=\frac{\sqrt{n+1}}{2}.
\end{align}
\end{lemma}

\begin{proof}
We give a proof for the reader's convenience.
Let $M=M_{DRT_n}$. Then by the definition, it holds that
\begin{equation}
MM^t=\frac{n+1}{4}I_n+\frac{n-3}{4}J_n.
\end{equation} 
Since $M+M^t=J_n-I_n$, we obtain the following equality.
\begin{equation}
M^2+M+\frac{n+1}{4}I_n-\frac{n+1}{4}J_n=O.
\end{equation} 
Since $DRT_n$ is regular, we see that $(n-1)/2$ is an eigenvalue of $M$ and a corresponding eigenvector is the all-one eigenvector $\boldsymbol{1}$.
Since $DRT_n$ is normal, each eigenvalue $\theta$ except for $(n-1)/2$ has an eigenvector $\boldsymbol{v}$ which is orthogonal to $\boldsymbol{1}$.
Thus, 
\begin{equation}
\Bigl(\theta^2+\theta+\frac{n+1}{4} \Bigr)\boldsymbol{v}=\boldsymbol{0}.
\end{equation} 
Since $\boldsymbol{v}\neq \boldsymbol{0}$, we get 
\begin{equation}
\Bigl(\theta^2+\theta+\frac{n+1}{4} \Bigr)=0,
\end{equation} 
completing the proof. 
\end{proof}

\begin{remark}
We remark that Corollary~\ref{cor:dr} is a generalization of the bound (\ref{eq:paley}) because Paley tournaments are also doubly-regular tournaments.
For other non-isomorphic examples of doubly regular tournaments, see e.g. \cite{IO94} and \cite{S69}. 
As shown in, for example, \cite{HW14} and \cite{RB72}, there are some known constructions of doubly regular tournaments such that the number of vertices is non-prime (and non-prime power). 
Especially, constructions of complex codebooks in \cite{HW14} provide $DRT_n$ for every integer $n$ such that each prime factor $f$ of $n$ is the form of $f \equiv 3 \pmod{4}$. 
\end{remark}
\begin{remark}
By the definition of $DRT_n$, $n$ must be a positive integer of the form $n \equiv 3 \pmod{4}$.
On the other hand, as an analogue of $DRT_n$ for integers $n$ of the form $n\equiv 1 \pmod{4}$, Savchenko~\cite{S16} introduced the notion of a nearly-doubly-regular tournament
$CNDR_n$ with $n$ vertices which is a certain regular tournament with exactly four eigenvalues distinct to $(n-1)/2$ with multiplicity $(n-1)/4$.
According to \cite{S16}, it holds that $\lambda(CNDR_n)=(\sqrt{n}+1)/2$.
Thus if there exists $CNDR_n$ for infinitely many $n \equiv 1\pmod{4}$, then it holds that
\[
C(CNDR_n)\leq \frac{1}{2}\binom{n}{2}+O(n^{\frac{3}{2}}\log n).
\]
It is conjectured in \cite{S16} (see also \cite{S17}) that there exists a $CNDR_n$ for every $n \equiv 1 \pmod{4}$.
Interestingly, Savchenko~\cite{S16} also found examples of $CNDR_p$ for primes $p=5, 13, 29, 53, 173, 229, 293$ and $733$ from the class of $T_p^4(S_{(0, 1)})$ in the first example, and thus Lemma~\ref{lem:eigen} can be improved for these examples.
(It is shown in \cite{S16} that for every prime $p \equiv 5 \pmod{8}$, $T_p^4(S_{(0, 1)})$ has exactly four eigenvalues distinct to $(p-1)/2$ with multiplicity $(p-1)/4$.)
It would be interesting to prove or disprove the existence of infinitely many primes $p \equiv 5 \pmod{8}$ such that the tournament $T_p^4(S_{(0, 1)})$ is in the class of $CNDR_p$.
\end{remark}

The third example is based on a construction of pseudo-random graphs due to Shparlinski~\cite{S08}.
For related facts on eliptic curves, see \cite[Section 2.1]{S08}.
For a prime $p$, let $n \in [p+1-2\sqrt{p}, p+1+2\sqrt{p}]$ be an odd integer.
It is known (e.g. \cite{BS04}, \cite{D41}) that there exists an eliptic curve $E$ over $\mathbb{F}_p$ such that the number of $\mathbb{F}_p$-rational points of $E$ is $n$.
It is also known (e.g. \cite{S95}) that all $\mathbb{F}_p$-rational points of $E$ form an abelian group $G$ of order $n$ under an operation $\oplus$.
Let $0_G$ be the identity of $G$.
For an element $s \in G$ and a subset $S \subset G$, the inverse of $s$ is denoted by $\ominus s$ and let $\ominus S=\{\ominus s \mid s \in S\}$.

\begin{definition}
Let $S \subset G$ be a subset such that $S \cup \ominus S \cup \{0_G \}=G$ and $|S|=(n-1)/2$. 
Then the tournament $T_{p,n}(S)$ is defined as follows.
\begin{equation}
\begin{split}
&V(T_{p,n}(S))=G, \\
&E(T_{p,n}(S))=\{(x, y)\in G^2 \mid x \ominus y \in S \}.
\end{split}
\end{equation}
\end{definition}
By the definition, $T_{p,n}(S)$ is a regular tournament with $n$ vertices.

\begin{corollary}
\label{cor:eliptic}
There exists a subset $S \subset G$ such that
\begin{equation}
C(T_{p,n}(S)) \leq \frac{1}{2}\binom{n}{2}+O(n^{\frac{3}{2}} \log^2 n).
\end{equation}
\end{corollary}

Corollary~\ref{cor:eliptic} is obtained by Lemma~\ref{cor:exp2} and the following evaluation of $\lambda(T_{p,n}(S))$ which follows from \cite[Theorem 1]{S08}.
\begin{lemma}[\cite{S08}]
There exists a subset $S \subset G$ such that
\label{lem:eigen3}
\begin{align}
\lambda(T_{p,n}(S))=O(\sqrt{n}\log n).
\end{align}
\end{lemma}
For the details of a construction of such a subset $S$, see \cite{S08}.

\begin{remark}
It is worth noting that as shown in \cite{S08}, almost all positive integers are in the interval $[p+1-2\sqrt{p}, p+1+2\sqrt{p}]$ for some prime $p$. 
Indeed, it holds (\cite{S08}) that
\[
\lim_{N \to \infty}\frac{|\{n \leq N \mid \text{$\exists$ prime $p$ s.t. $n$ is odd and $n \in [p+1-2\sqrt{p}, p+1+2\sqrt{p}]$}\}|}{\lceil\frac{N}{2} \rceil}=1.
\]
Thus the third example provides regular tournaments $T$ with $n$ vertices and small $\lambda(T)$ for almost all positive integers $n$. 
\end{remark}

\section{Sh\"{u}tte's problem for tournaments}
\label{sec:Shutte}
At last, in this section, we focus on another random-like property. 
\begin{definition}
\label{def:ec}
Let $k$ be a positive integer. 
A tournament $T$ has the property $S_k$ if for every $A \subset V(T)$ of size $k$, 
there exists a vertex $z \notin A$ directing to all members of $A$. 
\end{definition}
The {\it Sh\"{u}tte's problem} asks the existence of tournaments satisfying this property (see \cite{E63} and \cite{M68}).
As shown by Erd\H{o}s~\cite{E63}, random tournaments a.a.s. satisfy $S_k$ for any $k\geq 1$.
On the other hand, the problem of explicit constructions has been considered in graph theory. 
For example, Graham-Spencer~\cite{GS71} showed that the Paley tournament $T_p$ satisfies $S_k$ if $p>k^22^{2k-2}$ for each $k \geq 1$.
From the digraphs constructed in \cite{AHLNVW15}, we can also construct tournaments satisfying $S_k$ for every $k$ by adding some edges. 
At present, there seems to be almost no explicit constructions of tournaments satisfying both of the quasi-random property and $S_k$ except for Paley tournaments.
The following proposition and Corollary~\ref{cor:Tpm} show that the tournament $T_p^m(S_{\boldsymbol{i}})$ has the quasi-random property and $S_k$.

\begin{proposition}
\label{prop:shutte1}
Let $m$ be an even positive integer. 
Then for every $k \geq 1$, there exists a prime $p_m(k)$ such that for every prime $p>p_m(k)$, the tournament $T_p^m(S_{\boldsymbol{i}})$ has the property $S_k$.
\end{proposition}
Proposition~\ref{prop:shutte1} is proved by a direct generalization of the discussion in \cite{GS71} and \cite{AC06}, so we omit the proof here.
Moreover, it is not so hard to prove that $T_p^m(S_{\boldsymbol{i}})$ has the existentially closed property (see e.g. \cite{B09}).
\par We also note that doubly regular tournaments constructed in \cite{S69} satisfy both of the quasi-random property and $S_2$, which follows from Corollary~\ref{cor:dr} and the corollary in \cite[p.277]{S69}. 

\section*{Acknowledgement}
We would like to thank Masanori Sawa and Yujie Gu for their valuable comments. 
We also greatly appreciate Sergey Savchenko for his helpful remarks.


\end{document}